\documentclass[12pt]{amsart}

\usepackage{amssymb}

\newtheorem{theorem}{Theorem}[section]
\newtheorem{lemma}[theorem]{Lemma}
\newtheorem{proposition}[theorem]{Proposition}

\numberwithin{equation}{section}

\begin{document}
\baselineskip=15pt

\title[Principal bundles over finite fields]{Principal
bundles over finite fields}

\author[I. Biswas]{Indranil Biswas}

\address{School of Mathematics, Tata Institute of Fundamental
Research, Homi Bhabha Road, Bombay 400005, India}

\email{indranil@math.tifr.res.in}

\author[S. Subramanian]{S. Subramanian}

\address{School of Mathematics, Tata Institute of Fundamental
Research, Homi Bhabha Road, Bombay 400005, India}

\email{subramnn@math.tifr.res.in}



\date{}

\begin{abstract}
Let $M$ be an irreducible smooth projective variety defined
over $\overline{{\mathbb F}_p}$. Let $\varpi(M, x_0)$ be the
fundamental group scheme of $M$ with respect
to a base point $x_0$. Let $G$ be a connected
semisimple linear algebraic group over $\overline{{\mathbb F}_p}$.
Fix a parabolic subgroup $P\, \subsetneq\, G$, and also fix
a strictly anti--dominant character $\chi$ of $P$.
Let $E_G\, \longrightarrow\, M$ be a principal $G$--bundle
such that the associated line bundle $E_G(\chi)\, \longrightarrow\,
E_G/P$ is numerically effective.
We prove that $E_G$ is given by a homomorphism
$\varpi(M, x_0)\,\longrightarrow\, G$. As a consequence,
there is no principal $G$--bundle $E_G\, \longrightarrow\, M$
such that $\text{degree}(\varphi^*E_G(\chi)) \, >\, 0$
for every pair $(Y\, ,\varphi)$,
where $Y$ is an irreducible smooth projective curve, and
$\varphi: Y\longrightarrow E_G/P$ is a nonconstant morphism.
\end{abstract}

\maketitle

\section{Introduction}\label{sec1}

We recall a question of S. Keel: Let $X$ be a smooth projective
surface over $\overline{{\mathbb F}_p}$ such that $L.C\, >\, 0$
for every complete curve $C$. The question of Keel asks whether
$L$ is ample. (See \cite[p. 3959, Question 0.9]{Ke}.)

If the base field is complex numbers, then this question has a
negative answer as shown by Mumford \cite[p. 56, Example 10.6]{Ha1}.
Further examples were constructed in \cite{Su1} and \cite{MS} for
higher dimensional complex projective varieties and also for varieties 
defined over fields of positive characteristics (these fields are
not countable).

The varieties in all these examples are total spaces of flag bundles 
associated to vector bundles on curves, and the line bundles are
the naturally associated ones. Proposition \ref{prop0}
shows that these type of varieties and line 
bundles never produce examples that would give a negative answer
to the question of Keel.

Let $M$ be an irreducible smooth projective variety defined
over $\overline{{\mathbb F}_p}$.
Let $\varpi(M,\, x_0)$ be the fundamental group scheme
with respect to a base point $x_0\,\in\, M$.
The fundamental group scheme was introduced by
Nori \cite{No1}, \cite{No2}. Let $G$ be a connected
semisimple linear algebraic group over $\overline{{\mathbb F}_p}$.
Fix a parabolic proper subgroup $P\, \subset\, G$, and also fix
a strictly anti--dominant character $\chi$ of $P$.
Given a principal $G$--bundle $E_G\, \longrightarrow\, M$,
the quotient map $E_G\, \longrightarrow\, E_G/P$ defines
a principal $P$--bundle. Let $E_G(\chi)\, \longrightarrow\,
E_G/P$ be the line bundle associated to this principal $P$--bundle
for the character $\chi$.

We prove the following theorem (see Theorem \ref{thm1}):

\begin{theorem}\label{thm0}
Let $E_G$ be a principal $G$--bundle over $M$
such that the line bundle
$$
E_G(\chi)\,\longrightarrow\, E_G/P
$$
is numerically effective.
Then $E_G$ is given by a homomorphism
$\varpi(M, x_0)\,\longrightarrow\, G$.
\end{theorem}

The following proposition is proved using Theorem \ref{thm0}
(see Proposition \ref{prop3}):

\begin{proposition}\label{prop0}
There is no principal $G$--bundle $E_G$ over $M$
such that the line bundle $E_G(\chi)\, \longrightarrow\,
E_G/P$ has the following property: for every pair of the
form $(Y\, ,\varphi)$,
where $Y$ is an irreducible smooth projective curve, and
$\varphi\,:\, Y\,\longrightarrow\, E_G/P$
is a nonconstant morphism, the inequality
$$
{\rm degree}(\varphi^*E_G(\chi)) \, >\, 0
$$
holds.
\end{proposition}

\section{Preliminaries}\label{sec2}

Fix a prime $p$.
Let $M$ be an irreducible smooth projective variety
of dimension $d$ defined over $\overline{{\mathbb F}_p}$. Let
$$
F_M\, :\, M\, \longrightarrow\, M
$$
be the absolute Frobenius morphism of $M$. For any integer
$n\, \geq\, 1$, let
$$
F^n_M \, :=\, \overbrace{F_M\circ\cdots\circ F_M}^{n\mbox{-}
\rm{times}}\, :\, M\, \longrightarrow\, M
$$
be the $n$--fold iteration of $F_M$; by $F^0_M$ we will
denote the identity morphism of $M$.

For any $c\, \in\, {\rm CH}^d(M)$, let
$$
[c]\,\in\, {\mathbb Z}
$$
be the degree of $c$ ($d\,=\, \dim M$).

Fix a very ample line bundle ${\mathcal O}_M(1)$ on $M$. The degree of a
torsionfree coherent sheaf $V$ on $M$ is defined to
be
$$
\text{degree}(V)\,:=\, [c_1(V)\cdot c_1({\mathcal O}_M(1))^{d-1}]\, 
\in\,
{\mathbb Z}\, .
$$
We note that
\begin{equation}\label{f1}
\text{degree}(V)\,=\, \text{degree}(V\vert_D)\, ,
\end{equation}
where $D$ is a smooth complete intersection curve obtained
by intersecting $(d-1)$ hyperplanes on $M$ from the complete linear 
system $\vert {\mathcal O}_M(1)\vert$.
Let $U\, \subset\, M$ be a Zariski open subset such that
the codimension of the complement $M\setminus U$ is at
least two. For a torsionfree coherent sheaf $E\,\longrightarrow\,
U$, define
$$
\text{degree}(E)\, =\, \text{degree}(\iota_*E)\, ,
$$
where $\iota\, :\, U\, \hookrightarrow\, M$ is the inclusion map;
note that $\iota_*E$ is a coherent sheaf.

A Zariski open subset $U\, \subset\, M$ such that
the codimension of the complement $M\setminus U$ is at
least two will be called a \textit{big open subset}.

Let $G$ be a connected semisimple linear algebraic
group defined over $\overline{{\mathbb F}_p}$.
Let $E_G$ be a principal $G$--bundle over $M$.
Let $(Q\, ,U\, , \sigma)$ be a triple, where
\begin{itemize}
\item $Q \,\subset\, G$ is a maximal
proper parabolic subgroup,

\item $U \,\subset\, M$ is a big open subset, and

\item $\sigma \,:\, U \, \longrightarrow\, E_G/Q$ is a
reduction of the structure group, over $U$, to the subgroup $Q$.
\end{itemize}
The principal $G$--bundle $E_G$ is called
\textit{semistable} if for all such triples,
the inequality
\begin{equation}\label{c2}
\text{degree}(\sigma^* T_{\text{rel}})\, \geq \, 0
\end{equation}
holds, where $T_{\text{rel}}\,
\longrightarrow\, E_G/Q$ is the
relative tangent bundle for the projection $E_G/Q \,
\longrightarrow\, M$. (See \cite{Ra}, \cite{RR}, \cite{RS}.)

A principal $G$--bundle $E_G\, \longrightarrow\, M$ is called
\textit{strongly semistable} if the principal $G$--bundle
$(F^n_M)^*E_G$ is semistable for all $n\, \geq\, 0$.

We recall that a character $\chi$ of a parabolic subgroup
$P\, \subset\, G$ is called \textit{strictly anti--dominant}
if the associated line bundle
\begin{equation}\label{l.b}
G(\chi)\,:=\, G\times^P \overline{{\mathbb F}_p}
\, \longrightarrow\, G/P
\end{equation}
is ample.

\section{Restriction of principal bundles to curves}

We start with a simple lemma.

\begin{lemma}\label{lem0}
Let $E_G\, \longrightarrow\, M$ be a
principal $G$--bundle such that
for every pair of the form $(Y\, ,\varphi)$,
where $Y$ is an irreducible smooth projective curve, and
$\varphi\,:\, Y\,\longrightarrow\, M$
is a morphism, the principal $G$--bundle
$\varphi^*E_G\, \longrightarrow\,Y$ is semistable.
Then $\varphi^*E_G$ and $E_G$ are strongly semistable.
\end{lemma}

\begin{proof}
Let $F_C$ be the absolute Frobenius morphism of $C$. Replacing
$\varphi$ by $\varphi\circ F^n_C$ in the given condition,
we conclude that $(\varphi\circ F^n_C)^*E \,=\, (F^n_C)^*(\varphi^*E)$
is semistable. Hence $\varphi^*E$ is strongly semistable. Now from
\eqref{f1} it follows that $E$ is strongly semistable.
\end{proof}

\begin{proposition}\label{prop1}
Let $E_G\, \longrightarrow\, M$ be a
principal $G$--bundle such that
for every pair of the form $(C\, ,\varphi)$,
where $C$ is an irreducible smooth projective curve, and
$$
\varphi\,:\, C\,\longrightarrow\, M
$$
is a morphism, the principal $G$--bundle
$\varphi^*E_G\, \longrightarrow\,C$ is semistable.
Then
$$
[c_2({\rm ad}(E_G))\cdot c_1({\mathcal O}_M(1))^{d-2}] \, =\, 0\, .
$$
\end{proposition}

\begin{proof}
We will first prove this for vector bundles.
Let $E\, \longrightarrow\, M$ be a vector bundle
such that the pulled back vector bundle
$\varphi^*E\, \longrightarrow\, C$ is semistable
for every pair $(C\, ,\varphi)$ of the above type.

{}From Lemma \ref{lem0} we know that $\varphi^*E\, \longrightarrow
\, C$ is strongly semistable. Hence the endomorphism bundle
$End(\varphi^*E)\,=\, \varphi^*End(E)$ is semistable
\cite[p. 288, Theorem 3.23]{RR}. From this it follows that
$End(E)$ is numerically effective. To prove this, take any
morphism
$$
\gamma\, :\, C\, \longrightarrow\, {\mathbb P}(End(E))
$$
and set $\varphi\,=\, h\circ\gamma$, where $h:
{\mathbb P}(End(E))\longrightarrow M$ is the natural projection.
Note that $\gamma^*{\mathcal O}_{{\mathbb P}(End(E))}(1)$
is a quotient of $\varphi^*End(E)$. Since $\varphi^*End(E)$ is 
semistable of degree zero, we have
$$
\text{degree}(\gamma^*{\mathcal O}_{{\mathbb 
P}(End(E))}(1))\, \geq\, 0\, ,
$$
proving that $End(E)$ is numerically effective.

If $M$ is a curve then there is nothing to prove.
Take a smooth complete intersection surface
\begin{equation}\label{liota}
\iota\, :\, S\, \hookrightarrow\, M
\end{equation}
obtained by intersecting $(d-2)$ hyperplanes on $M$ from the complete 
linear system $\vert {\mathcal O}_M(1)\vert$. If $M$ is a surface,
then take $S\,=\, M$. Let
$$
W\, :=\, \iota^*End(E)
$$
be the restriction of $End(E)$ to $S$. We note that
\begin{equation}\label{l5}
[c_2(End(E))\cdot c_1({\mathcal O}_M(1))^{d-2}]\,=\, [c_2(W)]
\,\in\, {\mathbb Z}\, .
\end{equation}
Also, note that $c_1(W)\, =\, 0$ because $W\, =\, W^*$.

Substituting $W$ for $E$ in Lemma \ref{lem0} we conclude that
$W$ is strongly semistable.
Therefore, we have the Bogomolov inequality for $W$
\begin{equation}\label{eqbo}
[c_2(W)] \, \geq\, 0
\end{equation}
(see \cite{Bo}, \cite{La}); recall that $c_1(W)\,=\, 0$.

Let
\begin{equation}\label{eqf1}
f\, :\, Z\, :=\, {\mathbb P}(\iota^*End(E))
\, \longrightarrow\, S
\end{equation}
be the projective bundle parametrizing the hyperplanes
in the fibers of $\iota^*End(E)\, =\, W$.
The tautological line bundle 
${\mathcal O}_Z(1)$ over $Z$ will be denoted by $L$.

The Grothendieck's construction of Chern classes gives
$$
[f^*c_2(W)\cdot c_1(L)^{r-2}] - [f^*c_1(W)\cdot c_1(L)^{r-1}] +
[c_1(L)^{r}]\, =\, 0
$$
(see \cite[page 429]{Ha2}), where $r\,=\, \text{rank}(W)$,
and $f$ is the projection in \eqref{eqf1} to the surface $S$.
Therefore,
\begin{equation}\label{eqf3}
[f^*c_2(W)\cdot c_1(L)^{r-2}] \,+\, [c_1(L)^{r}] \,=\, 0\, ,
\end{equation}
because $c_1(W)\,=\,0$.

The line bundle $L\, \longrightarrow \,{\mathbb P}(W)$ is numerically 
effective because it is a restriction of the numerically effective
line bundle ${\mathcal O}_{{\mathbb P}(End(E))}(1)$.
Consequently, $[c_1(L)^{r}]\, \geq\, 0$.

Since $[c_1(L)^{r}]\, \geq\, 0$, from \eqref{eqf3}
we conclude that $[c_2(W)] \, \leq\, 0$. Comparing this
with \eqref{eqbo} we conclude that
$[c_2(W)] \, =\, 0$.
Hence from \eqref{l5} we conclude that
\begin{equation}\label{l.c2}
[c_2(End(E))\cdot c_1({\mathcal O}_M(1))^{d-2}] \, =\, 0\, .
\end{equation}

Now we consider the general case of principal $G$--bundles.
Let $E_G$ be as in the statement of the proposition.

Let $m$ be the dimension of $G$. Let ${\rm ad}(E_G)$ be the
adjoint vector bundle. The line bundle
$\bigwedge^m {\rm ad}(E_G)$ is trivial because the adjoint
action of $G$ on $\bigwedge^m \text{Lie}(G)$ is trivial.
In particular, $c_1({\rm ad}(E_G))\,=\, 0$, and hence
\begin{equation}\label{ch1}
c_2(End({\rm ad}(E_G)))\,=\, 2m\cdot c_2({\rm ad}(E_G))\, .
\end{equation}

Since $\varphi^*E_G$ is strongly semistable (Lemma \ref{lem0}),
the adjoint vector bundle $\text{ad}(\varphi^*E_G)\,=\,
\varphi^*\text{ad}(E_G)$ is semistable
\cite[p. 288, Theorem 3.23]{RR}. Hence from \eqref{l.c2}
we know that
$$
[c_2(End({\rm ad}(E_G)))\cdot c_1({\mathcal O}_M(1))^{d-2}] \, =\, 0\, .
$$
Therefore, \eqref{ch1} implies that $[c_2({\rm ad}(E_G))
\cdot c_1({\mathcal O}_M(1))^{d-2}] \, =\, 0$.
This completes the proof of the proposition.
\end{proof}

\section{Fundamental group scheme and principal bundles}

Fix a base point $x_0\, \in\, M$. Let
$\varpi(M,\, x_0)$ be the fundamental group scheme
\cite{No1}, \cite{No2}. There is a universal principal
$\varpi(M,\, x_0)$--bundle
$$
F_{\varpi(M, x_0)}\, \longrightarrow\, M\, .
$$
Given a homomorphism $\rho\, :\, \varpi(M,\, x_0)\, \longrightarrow
\, G$, we get a principal $G$--bundle $F_{\varpi(M, x_0)}(G)$
over $M$ by extending the structure group of $F_{\varpi(M, x_0)}$
using $\rho$.

Fix a parabolic subgroup $P\, \subsetneq\, G$. Fix a strictly
anti--dominant character $\chi$ of $P$.

Given a principal $G$--bundle $E_G\, \longrightarrow\, M$,
the quotient map $E_G\, \longrightarrow\, E_G/P$ defines
a principal $P$--bundle. Let $E_G(\chi)\, \longrightarrow\,
E_G/P$ be the line bundle associated to this principal $P$--bundle
for the character $\chi$.

\begin{theorem}\label{thm1}
Let $E_G$ be a principal $G$--bundle over $M$
such that the line bundle $E_G(\chi)$ over $E_G/P$ is
numerically effective. Then $E_G$ is given by a homomorphism
$\varpi(M, x_0)\,\longrightarrow\, G$.
\end{theorem}

\begin{proof}
Take any pair $(C\, ,\theta)$, where $C$ is an irreducible
smooth projective curve, and
$$
\theta\,:\, C\,\longrightarrow\, M
$$
is a morphism. Consider the fiber bundle
$$
E_G/P\,\longrightarrow \,M\, .
$$
Note that $\theta^*(E_G/P)\,=\,(\theta^*E_G)/P$,
and the pullback of the line bundle
$E_G(\chi)$ to $\theta^*(E_G/P)$ coincides
with the line bundle $(\theta^*E_G)(\chi)$ associated to the
principal $P$--bundle $\theta^*E_G\, \longrightarrow\,
(\theta^*E_G)/P$ for the character $\chi$. Since $E_G(\chi)$ is
numerically effective, and the pullback of a numerically effective
line bundle is numerically effective, we conclude that the line bundle
$$
(\theta^*E_G)(\chi)\, \longrightarrow\, (\theta^*E_G)/P
$$
is numerically effective. This implies that the principal
$G$--bundle $\theta^*E_G$ is semistable \cite[p. 766, Theorem 3.1]{BP}.

Therefore, from Lemma \ref{lem0} and
Proposition \ref{prop1} we conclude that $E_G$ is strongly semistable,
and
$$
[c_2({\rm ad}(E_G))\cdot c_1({\mathcal O}_M(1))^{d-2}] \, =\, 0\, .
$$
Hence $E_G$ is a given by a homomorphism
$\varpi(M,\, x_0)\, \longrightarrow\, G$
\cite[pp. 210--211, Theorem 1.1]{Bi} (for vector
bundles this was proved earlier in \cite{Su2}). This completes the
proof of the theorem.
\end{proof}

\begin{proposition}\label{prop3}
There is no principal $G$--bundle $E_G$ over $M$
such that the line bundle $E_G(\chi)\, \longrightarrow\,
E_G/P$ has the following property: for every pair of the
form $(Y\, ,\varphi)$,
where $Y$ is an irreducible smooth projective curve, and
$\varphi\,:\, Y\,\longrightarrow\, E_G/P$
is a nonconstant morphism, the inequality
${\rm degree}(\varphi^*E_G(\chi)) \, >\, 0$ holds.
\end{proposition}

\begin{proof}
Let $E_G\, \longrightarrow\, M$ be a principal $G$--bundle
such that
\begin{equation}\label{ch}
{\rm degree}(\varphi^*E_G(\chi)) \, >\, 0
\end{equation}
for every pair $(Y\, ,\varphi)$ of the above type. From
Theorem \ref{thm1} we know that $E_G$ is given by
a homomorphism
\begin{equation}\label{rho}
\rho\, :\, \varpi(M,\, x_0)\, \longrightarrow\, G
\end{equation}

Fix a faithful representation
$$
\eta\, :\, G\, \hookrightarrow\, \text{GL}(V_0)\, .
$$
Let $E_G(V_0)\,:=\, E_G\times^G V_0\, \longrightarrow\, M$
be the associated vector bundle. Since $E_G$ is given by
the homomorphism $\rho$ in \eqref{rho}, the vector bundle
$E_G(V_0)$ is given by the homomorphism
$$
\eta\circ\rho\, :\, \varpi(M,\, x_0)\, \longrightarrow\,
\text{GL}(V_0)\, .
$$
In particular, $E_G(V_0)$ is an essentially finite vector
bundle \cite{No1}, \cite{No2}. Therefore, there is a finite
morphism
$$
f\, :\, \widetilde{M}\, \longrightarrow\, M\, ,
$$
where $\widetilde{M}$ is an irreducible smooth projective
variety of dimension $d$, such that the pulled back vector
bundle $f^*E_G(V_0)$ is trivial \cite[p. 557]{BH}.

Since $f^*E_G(V_0)$ is trivial, the principal $\text{GL}(V_0)$--bundle
$(f^*E_G)\times^G \text{GL}(V_0)\, \longrightarrow\, \widetilde{M}$,
which is the extension of structure group of $f^*E_G$ by $\eta$,
is trivial. So the reduction of structure group
$$
f^*E_G\, \hookrightarrow\, (f^*E_G)\times^G \text{GL}(V_0)\,=\,
\widetilde{M}\times \text{GL}(V_0)
$$
is given by a morphism $\widetilde{M}\, \longrightarrow\,
\text{GL}(V_0)/G$. Since $G$ is semisimple, the quotient
$\text{GL}(V_0)/G$ is an affine variety, hence there is no
nonconstant morphism to it from $\widetilde{M}$. Therefore,
the principal $G$--bundle $f^*E_G$ is trivializable. Fix
a trivialization of it.

Fixing a point $z_0$ of $G/P$, we construct the constant section
passing through $z_0$
$$
\sigma\, :\, \widetilde{M}\, \longrightarrow\, (f^*E_G)/P\,=\,
\widetilde{M}\times (G/P)
$$
of the projection $(f^*E_G)/P\,\longrightarrow\, \widetilde{M}$.
The composition $\beta\circ\sigma\, :\, \widetilde{M}\, 
\longrightarrow\, E_G/P$, where
$$
\beta\, :\, (f^*E_G)/P \,=\, f^*(E_G/P)\,
\longrightarrow\, E_G/P
$$
is the natural map, has the property that the line bundle
$$
(\beta\circ\sigma)^*E_G(\chi)\, \longrightarrow\,
\widetilde{M}
$$
is trivial. Indeed, $\beta^*E_G(\chi)$ is the pullback
of the line bundle $G(\chi)$ (defined in \eqref{l.b})
to $(f^*E_G)/P\,=\, \widetilde{M}\times (G/P)$
by the natural projection. Since $(\beta\circ\sigma)^*E_G(\chi)$
is trivial, the assumption that $E_G(\chi)$ satisfies the
condition in \eqref{ch} is contradicted.
This completes the proof of the proposition.
\end{proof}

%%%%%%%%%%%%%%%%%%%%%%%%%%%%%%%%%%%%%%%%%%%%%%%%%%%%%%%%%%%%%%%%%%%

\end{document}